\documentclass[12pt,a4paper]{amsart}
\usepackage{mathrsfs}
\usepackage{amsfonts}
\usepackage{txfonts}

\usepackage{hyperref}
\usepackage{latexsym}
\usepackage{amssymb}

\usepackage{enumitem}

\newtheorem{theorem}{Theorem}[section]
\newtheorem{lemma}[theorem]{Lemma}

\newtheorem{problem}[theorem]{Problem}

\theoremstyle{definition}

\newtheorem{remark}[theorem]{Remark}

\newcommand{\eps}{\varepsilon}

\newcommand{\Lip}{\operatorname{Lip}}
\newcommand{\F}{\mathcal{F}}
\numberwithin{equation}{section}

\begin{document}

\title[isometric embedding]
{\bf $\varepsilon$-isometries without isometric embeddings}

\author{ Longfa Sun$^\ast$, Yipeng Zhang }

\address{Longfa Sun: Hebei Key Laboratory of Physics and Energy Technology, School of Mathematics and Physics, North China Electric Power University, Baoding, 071003, China}
 \email{sun.longfa@ncepu.edu.cn}

\address{Yipeng Zhang: School of Mathematical Sciences, Xiamen University, Xiamen, 361005, China}
\email{19020250157654@stu.xmu.edu.cn}

\thanks{$^\ast$ Corresponding author.\\This work is supported by  the Fundamental Research Funds for the Central Universities (Grant no. 2025MS178). }

\date{}

\begin{abstract}
We show that there exist two separable real Banach spaces $X$ and $Y$
such that, for every $\varepsilon>0$, there is a standard exact $\varepsilon$-isometry
$f_\varepsilon: X\to Y$ whose distortion of every distance is between $0$
and $\varepsilon$, whereas no isometric embedding of $X$ into $Y$ exists,
even without a linearity assumption. This gives a negative answer to \cite[Problem 1 and Problem 2]{cheng4} of Cheng and Zhou. The construction combines a
snowflaked Lipschitz-free space with the Schur-property theorem of Kalton
and the isometric linearization theorem of Godefroy and Kalton.
\end{abstract}

\keywords{$\varepsilon$-isometries, isometries, Lipschitz-free spaces}

\subjclass{46B04, 46B20 }

\maketitle

\section{introduction}

Let $X$ and $Y$ be two real Banach  spaces. A map $f: X\rightarrow Y$ is called an $\varepsilon$-isometry for some $\varepsilon\geq0$ provided
\begin{equation*}
\big|\|f(x)-f(y)\|_Y-\|x-y\|_X\big|\leq\varepsilon,\;\forall\;x, y\in X.
\end{equation*}
$f$ is called standard if $f(0)=0$ and $0$-isometry is simply called isometry.

\emph{Isometries and linear isometries}\quad
The celebrated Mazur-Ulam theorem \cite{ma} states that every surjective standard isometry between real Banach spaces is necessarily linear.
For non-surjective isometries, Figiel \cite{figiel} extended Mazur-Ulam theorem by showing that: every standard isometry $f: X\rightarrow Y$ admits a linear left-inverse $T: \overline{span} (f(X))\rightarrow X$ of norm one, i.e., $T\circ f=Id_{X}$ with $\|T\|=1$.
Combining  Figiel's result and the Lipschitz-free spaces,
Godefroy and Kalton \cite{gode}  showed the following deep isometric linearization theorem, which resolved a long-standing open problem whether the existence of an isometry implies the existence of a linear isometry.
\begin{theorem}[Godefroy--Kalton]\label{thm:GK}
If $X$ is a separable Banach space and there is an isometry $f: X\rightarrow Y$, then $Y$ contains a linear subspace linearly isometric to $X$; for every nonseparable weakly compactly generated space $X$, there exists a Banach space $Y$ so that $X$ can be isometrically embedded into $Y$, but $X$ is not linearly isomorphic to any subspace of $Y$.
\end{theorem}

\emph{$\varepsilon$-Isometries and  isometries}
In 1945, Hyers and Ulam \cite{hy1} first investigated $\varepsilon$-isometries  and asked:
for every standard surjective $\eps$-isometry $f: X\rightarrow Y$, does there exist some linear surjective isometry $U: X\rightarrow Y$ such that  $f-U$ is  uniformly bounded on $X$?
After 50 years efforts of a number of mathematicians, the following sharp estimate was finally obtained by Omladi\v{c} and \v{S}emrl \cite{om}.
\begin{theorem} [Omladi\v{c}-\v{S}emrl]
If $f:X\rightarrow Y$ is a standard surjective $\varepsilon$-isometry, then there is a surjective linear isometry $U:X\rightarrow Y$ such that
\begin{equation*}
\|f(x)-U(x)\|_Y\leq2\varepsilon,\;\forall\; x\in X.
\end{equation*}
\end{theorem}
For non-surjective $\varepsilon$-isometries, Cheng and Zhou \cite{cheng4} showed that every standard  $\varepsilon$-isometry $f:X\rightarrow Y$ induces a linear closed subspace $N$ of $Y^\ast$ and a linear surjective isometry $U: X^{\ast\ast}\rightarrow N^\perp$.  Thus $X\subset X^{\ast\ast}$ is linearly isometric to a subspace of $Y^{\ast\ast}$. If $Y$ is reflexive, then $Y$ contains a linear isometric copy of $X$. For general Banach spaces, they raised the following two fundamental  and interesting problems.

\begin{problem}\label{problem1}\cite[Problem 1]{cheng4}
For two Banach spaces $X$, $Y$, if there is an $\varepsilon$-isometry $f:X\rightarrow Y$ for some
$\varepsilon>0$, does $Y$ contain an isometric (not necessarily linear) copy of $X$?
\end{problem}

\begin{problem}\label{problem2}\cite[Problem 2]{cheng4}
For two Banach spaces $X$, $Y$, if $X$ is separable and if there is an $\varepsilon$-isometry $f:X\rightarrow Y$ for some
$\varepsilon>0$, is $X$ linearly isometric to a subspace of $Y$?
\end{problem}
In this paper, we show the following theorem which gives a negative answer to Problem \ref{problem1} and Problem \ref{problem2}  in a stronger form which means that we can construct the $\varepsilon$-isometry such that it is injective uniformly continuous and its inverse on its
range is $1$-Lipschitz.

\begin{theorem}\label{thm:main}
There are fixed separable real Banach spaces $X$ and $Y$ such that the
following assertions hold.
\begin{enumerate}[label=\textup{(\roman*)}]
\item For every $\varepsilon>0$, there is an injective map
$f_\varepsilon:X\to Y$ satisfying
\begin{equation}\label{eq:eps}
 0\leq \|f_\varepsilon(x)-f_\varepsilon(y)\|_Y-
       \|x-y\|_X\leq \varepsilon,\;
 \forall\;x,y\in X.
\end{equation}
Moreover, $f_\varepsilon$ is uniformly continuous and its inverse on its
range is $1$-Lipschitz.
\item There is no isometric embedding $g:X\to Y$, whether or not $g$ is
assumed linear.
\end{enumerate}
One may take $X=(\ell_2,
\;\|\cdot\|_2)$ and
\[
 Y=\F_{\omega}(\ell_2),\qquad
 \omega(t)=\max\{t,\sqrt t\}.
\]
\end{theorem}

\section{preliminaries}

Let $(M,d)$ be a pointed metric space with a distinguished point (the origin) which we may
denote by $0$, $\Lip_0(M)$ is the Banach space of
all real-valued Lipschitz functions  $h: M\rightarrow \mathbb{R}$ with $h(0)=0$ under the norm
\[
 \|h\|_{\Lip}=\sup\left\{\frac{|h(u)-h(v)|}{d(u,v)}:\;u,v\in M,\;u\neq v\right\}.
\]

For $u\in M$, let $\delta_M(u)\in\Lip_0(M)^*$ be evaluation at $u$.  The
\emph{Lipschitz-free space} over $M$ is
\[
 \F(M)=\overline{\operatorname{span}}\{\delta_M(u):u\in M\}
 \subseteq \Lip_0(M)^*.
\]
The canonical map $\delta=\delta_M: M\to \F(M)$ is an isometric embedding and
$\delta(0)=0$.

If $(M,d)$ is pointed and $\omega:[0,\infty)\to[0,\infty)$ is such that
$d_\omega=\omega\circ d$ is a metric, we write
\[
 \F_\omega(M)=\F(M,d_\omega).
\]

Following Kalton's terminology,  a \emph{gauge} is a continuous, increasing,
subadditive function $\omega:[0,\infty)\to[0,\infty)$ such that
$\omega(0)=0$ and $\omega(t)\geq t$ for $0\leq t\leq1$.  It is
\emph{nontrivial} when
\[
 \lim_{t\downarrow0}\frac{\omega(t)}{t}=\infty.
\]
It is \emph{strongly normalized} when $\omega(t)=t$ for every $t\geq1$.

Recall that a Banach space $X$ has the \emph{Schur property} if every weakly
convergent sequence in $X$ converges in norm.

The following result is due to Kalton and is well-known as Kalton's Schur theorem.
\begin{theorem}\cite[Theorem~4.6]{kalton}
\label{thm:kalton}
If $M$ is any pointed metric space and $\omega$ is a nontrivial gauge,
then $\F_\omega(M)$ has the Schur property.
\end{theorem}

Note that the above theorem applies to unbounded metric spaces; no boundedness or
separability hypothesis is present in its statement. And, if $M$ is separable, then
$\F(M)$ is separable. Since for a countable dense subset $D\subseteq M$
containing $0$, the rational linear span of $\delta(D)$ is dense in
$\F(M)$.

\section{The counterexample}

Let
\[
\omega(t)=\max\{t,\sqrt t\},\;t\geq0,
\]
and $\tilde{X}=(\ell_2, \rho)$ be the space equipped with the pointed metric
\begin{equation*}
 \rho(x,y)=\omega(\|x-y\|_2),\;\forall\;x,y\in\ell_2.
\end{equation*}

\begin{lemma}\label{lem:gauge}
The function $\omega$ is a strongly normalized nontrivial gauge.  In
particular, $\rho$ is a metric of $\tilde{X}$.  The topologies induced by
the metrics $\rho$ and $d(x,y)=\|x-y\|_2$   coincide  on $\ell_2$.
\end{lemma}

\begin{proof}
The functions $t\mapsto t$ and $t\mapsto\sqrt t$ are continuous,
increasing, and subadditive on $[0,\infty)$.  The maximum of two
nonnegative increasing subadditive functions is again subadditive,
because
\[
 \max\{a_1+a_2,b_1+b_2\}
 \leq \max\{a_1,b_1\}+\max\{a_2,b_2\}.
\]
Thus $\omega$ is continuous, increasing, subadditive, and vanishes at
$0$.  It satisfies $\omega(t)\geq t$ for $0\leq t\leq1$, while
\[
 \frac{\omega(t)}t=\frac1{\sqrt t}\longrightarrow\infty,\;
 t\downarrow0.
\]
Moreover, $\omega(t)=t$ for $t\geq1$, so it is strongly normalized.
Subadditivity of $\omega(t)$ proves the triangle inequality for $\rho$. And hence $\rho$ is a metric.  Finally,
$\omega$ is continuous at $0$ and $\omega(t)\geq t$, which proves the
claim about the topology.
\end{proof}

Put
\[
 Y=\F(\tilde{X})=\F(\ell_2,\rho)=\F_\omega(\ell_2),
\]
and denote its canonical embedding by $\delta: \tilde{X}(=(\ell_2,\rho))\to Y$.  Thus
\begin{equation}\label{eq:delta-distance}
 \|\delta(u)-\delta(v)\|_Y
 =\rho(u,v)=\max\{\|u-v\|_2, \sqrt{\|u-v\|_2}\},\;\forall\;u,v\in \ell_2.
\end{equation}
By Lemma~\ref{lem:gauge} and Theorem~\ref{thm:kalton}, $Y$ has the
Schur property.  The space $\tilde{X}=(\ell_2,\rho)$ is separable because it has the
same topology as $\ell_2$, and hence $Y$ is separable as well.

\begin{proof}[Proof of Theorem~\ref{thm:main}]
Let $X=(\ell_2,
\;\|\cdot\|_2)$ and $Y=\F(\tilde{X})=\F(\ell_2,\rho)=\F_\omega(\ell_2)$.
Fix $\varepsilon>0$ and define
\begin{equation*}
 f_\varepsilon(x)=4\varepsilon\,
 \delta\!\left(\frac{x}{4\varepsilon}\right),\;\forall\;
 x\in\ell_2.
\end{equation*}
For $x,y\in X=(\ell_2,\;\|\cdot\|_2) $, put $r=\|x-y\|_2$.  Equation
\eqref{eq:delta-distance} gives
\begin{align}
 \|f_\varepsilon(x)-f_\varepsilon(y)\|_Y
 &=4\varepsilon\,
   \omega\!\left(\frac{r}{4\varepsilon}\right)\notag\\
 &=\max\{r,2\sqrt{\varepsilon r}\}.
 \label{eq:exact-distance}
\end{align}
Consequently the distance error is
\[
 \max\{0,2\sqrt{\varepsilon r}-r\}.
\]
Note that
\begin{equation}\label{eq:square}
 2\sqrt{\varepsilon r}-r
 =\varepsilon-(\sqrt r-\sqrt\varepsilon)^2
 \leq\varepsilon,
\end{equation}
this yields  \eqref{eq:eps}.  Its lower inequality also
shows that $f_\varepsilon$ is injective and that
$f_\varepsilon^{-1}:f_\varepsilon(X)\to X$ is $1$-Lipschitz.  Formula
\eqref{eq:exact-distance} shows directly that $f_\varepsilon$ is
uniformly continuous.

It remains to exclude an exact isometric embedding.  Suppose that
$g:X\to Y$ is  an isometric embedding.   Theorem~\ref{thm:GK} then yields a linear
isometric embedding
\[
 T:X\longrightarrow Y.
\]
Let $(e_n)$ be the standard unit vector basis of $X=(\ell_2,\;\|\cdot\|_2)$.  Since
$e_n\rightharpoonup0$ (weakly converge to $0$) in $X$, bounded linearity of $T$ gives
$Te_n\rightharpoonup0$ in $Y$.  The Schur property of $Y$ would imply
$\|Te_n\|_Y\to0$.  This is impossible because $T$ is an isometry and
$\|Te_n\|_Y=1$ for every $n$.  Hence no such $g$ exists.
\end{proof}

\begin{remark}
The map $f_\varepsilon$ is an exact $\varepsilon$-isometry, equality occurs in \eqref{eq:square} when
$\|x-y\|_2=\varepsilon$.
\end{remark}

\bibliographystyle{amsalpha}

\begin{thebibliography}{AcBeRu}

\bibitem {cheng4} L. Cheng, Y. Zhou, On perturbed metric-preserved mappings and their stability
characterizations, J. Funct. Anal. 266(2014), 4995-5015.
\bibitem {figiel} T. Figiel, On non linear isometric embeddings of normed linear spaces, Bull. Acad. Polon. Sci. Math. Astro. Phys.16 (1968),185-188.
\bibitem {gode} G. Godefroy, N.J. Kalton, Lipschitz-free Banach spaces, Studia Math. 159 (2003), 121-141.
\bibitem {hy1} D.H. Hyers, S. M. Ulam, On approximate isometries, Bull.
Amer. Math. Soc.51 (1945), 288-292.
\bibitem{kalton}
N.J. Kalton, Spaces of Lipschitz and H\"older functions and their applications,
Collect. Math. 55 (2004) (2), 171--217.
\bibitem {ma} S. Mazur, S. Ulam, Sur les transformations isom\'{e}triques d'espaces vectoriels norm\'{e}s, C.R. Acad. Sci. Paris 194(1932), 946-948.
\bibitem {om} M. Omladi\v{c}, P. \v{S}emrl, On non linear perturbations of
isometries, Math. Ann. 303 (1995), 617-628.









\end{thebibliography}

\end{document}